\documentclass[12pt]{amsart}
\title{Computing the Minimal Model for the Quantum Symmetric Algebra}
\author{Daniel Barter}
\date{\today}

\makeatletter
\def\@endtheorem{\endtrivlist}
\makeatother 

\usepackage{easyfig} 
\usepackage{float}
\usepackage{booktabs} 
\usepackage{amsmath,amssymb}
\usepackage{tikz-cd}
\usepackage{mathrsfs} 
\usepackage{tikz}
\usetikzlibrary{calc}

\usepackage[boxsize = 0.5em]{ytableau} 
\usepackage{enumerate} 
\usepackage[margin=2.5cm]{geometry} 

\theoremstyle{plain}
\newtheorem{Theorem}{Theorem}

\newtheorem{Proposition}[Theorem]{Proposition}

\theoremstyle{definition}
\newtheorem{Definition}[Theorem]{Definition}
\newtheorem{Lemma}[Theorem]{Lemma}
\newtheorem{Example}[Theorem]{Example}

\DeclareMathOperator{\End}{End}
\DeclareMathOperator{\Rep}{Rep}
\DeclareMathOperator{\Sym}{Sym}
\DeclareMathOperator{\QSym}{QSym}
\DeclareMathOperator{\GL}{GL}

\DeclareMathOperator{\coker}{coker}

\newcommand{\Vect}{{\bf Vec}}
\newcommand{\XX}{{\bf X}}

\newcommand{\DD}{{\bf D}}

\begin{document}

\begin{abstract}
In this note, we use some of the tensor categorial machinery developed by the quantum algebra community to study algebraic objects which appear in representation stability. In \cite{MR3430359}, Sam and Snowden prove that the twisted commutative algebra $\Sym$ is Morita equivalent to the horizontal strip category. Their proof relies on a lemma proved by Olver in \cite{MR924166}. We give a self contained proof that replaces Olver's lemma with information about the associator in the underlying category of polynomial $\GL(\infty)$-representations. In fact, we prove a quantum analogue of the theorem. The classical version follows by letting the parameter converge to $1$. 
\end{abstract}

\maketitle

\section{Introduction} \label{sec:introduction}

Let $\mathcal{S}$ be the category of polynomial $\GL(\infty)$-representations studied by Sam and Snowden in \cite{MR3430359}. This category contains the algebra $\Sym = \mathbb{C}[x_1,x_2,\dots]$ which is Morita equivalent to ${\bf FI}$, the category of finite sets with injections. A proof can be found in \cite{MR3556290}. In Section 3 of \cite{MR3430359}, Sam and Snowden prove that $\Sym$ is Morita equivalent to ${\bf HS}$, the category whose objects are partitions and whose morphisms are defined by
\[
{\bf HS}(\lambda,\mu) = \begin{cases}
\mathbb{C} \{ \mu \backslash \lambda \} & \lambda \subseteq \mu \; , \; \mu \backslash \lambda \in {\rm
HS}
\\
0 & \text{otherwise}
\end{cases}
\]
Composition is defined as follows: Assume that \( \mu \backslash \lambda \) and \( \nu \backslash \mu \) are horizontal strips. If \( \nu
\backslash \lambda \) is a horizontal strip, then
\[
(\nu \backslash \mu)(\mu \backslash \lambda) =  \nu \backslash \lambda.
\]
If \( \nu \backslash \lambda \) is not a horizontal strip, then the composition is zero. Now let $\mathcal{H}$ be the category of polynomial type 1 representations of $U_a(\mathfrak{gl}_{\infty})$ defined in Definition \ref{ex:main_example}. Inside $\mathcal{H}$, we have the quantum symmetric algebra $\QSym$. In this chapter, we prove the following:
\begin{Theorem} \label{thm:minimal_model}
The quantum symmetric algebra $\QSym$ is Morita equivalent to ${\bf HS}$ for generic $a$.
\end{Theorem}
\noindent Theorem \ref{thm:minimal_model} implies that many of the results in \cite{MR3430359} which hold for $\Sym$ are also true for $\QSym$. 

The author would like to thank Corey Jones, Scott Morrison, Andrew Snowden and Phil Tosteson for many useful conversations and their support.

\section{Preliminaries} \label{sec:Preliminaries}

\begin{Definition} \label{def:tree_string_diagrams}
 Let $\XX$ be a semi-simple tensor category. Index the simple objects with a set $\Lambda$. Choose a basis for each $\XX(\mu,\lambda \otimes \nu)$ denoted by
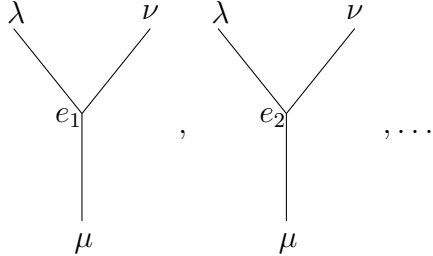
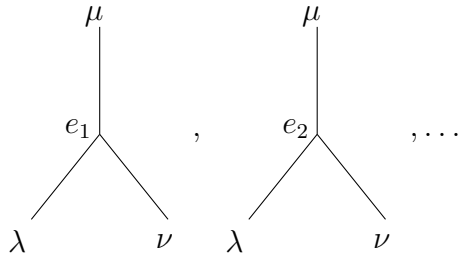
\[
\begin{tikzpicture}[yscale=-1,scale=0.015,baseline={([yshift=-.5ex]current bounding box.center)}]
\begin{scope}[shift={(0.00mm,719.29mm)}]
\draw [fill=none,draw=black] (887.14mm,53.08mm)
-- (281.43mm,-701.21mm)
;
\draw [fill=none,draw=black] (886.64mm,53.08mm)
-- (1492.36mm,-701.21mm)
;
\draw [fill=none,draw=black] (886.79mm,52.41mm)
-- ++(0.00mm,953.58mm)
;
\node [black] at (302.86mm,-850mm) { $\lambda$ };
\node [black] at (1494.29mm,-850mm) { $\nu$ };
\node [black] at (902.86mm,1200mm) { $\mu$ };
\node [black] at (768.57mm,86.65mm) { $e_1$ };
\end{scope}
\end{tikzpicture} \; , \;
\begin{tikzpicture}[yscale=-1,scale=0.015,baseline={([yshift=-.5ex]current bounding box.center)}]
\begin{scope}[shift={(0.00mm,719.29mm)}]
\draw [fill=none,draw=black] (887.14mm,53.08mm)
-- (281.43mm,-701.21mm)
;
\draw [fill=none,draw=black] (886.64mm,53.08mm)
-- (1492.36mm,-701.21mm)
;
\draw [fill=none,draw=black] (886.79mm,52.41mm)
-- ++(0.00mm,953.58mm)
;
\node [black] at (302.86mm,-850mm) { $\lambda$ };
\node [black] at (1494.29mm,-850mm) { $\nu$ };
\node [black] at (902.86mm,1200mm) { $\mu$ };
\node [black] at (768.57mm,86.65mm) { $e_2$ };
\end{scope}
\end{tikzpicture} \; , \dots
\]
\noindent and let
\[
\begin{tikzpicture}[yscale=-1,scale=0.015,baseline={([yshift=-.5ex]current bounding box.center)}]
\begin{scope}[shift={(0.00mm,719.29mm)}]
\draw [fill=none,draw=black] (888.57mm,251.04mm)
-- (282.86mm,1005.33mm)
;
\draw [fill=none,draw=black] (888.79mm,250.68mm)
-- (1494.50mm,1004.97mm)
;
\draw [fill=none,draw=black] (888.93mm,251.35mm)
-- ++(0.00mm,-953.58mm)
;
\node [black] at (157.14mm,1200.93mm) { $\lambda$ };
\node [black] at (1462.86mm,1200mm) { $\nu$ };
\node [black] at (840.00mm,-800.78mm) { $\mu$ };
\node [black] at (700mm,200mm) { $e_1$ };
\end{scope}
\end{tikzpicture} \; , \;
\begin{tikzpicture}[yscale=-1,scale=0.015,baseline={([yshift=-.5ex]current bounding box.center)}]
\begin{scope}[shift={(0.00mm,719.29mm)}]
\draw [fill=none,draw=black] (888.57mm,251.04mm)
-- (282.86mm,1005.33mm)
;
\draw [fill=none,draw=black] (888.79mm,250.68mm)
-- (1494.50mm,1004.97mm)
;
\draw [fill=none,draw=black] (888.93mm,251.35mm)
-- ++(0.00mm,-953.58mm)
;
\node [black] at (157.14mm,1200.93mm) { $\lambda$ };
\node [black] at (1462.86mm,1200mm) { $\nu$ };
\node [black] at (840.00mm,-800.78mm) { $\mu$ };
\node [black] at (700mm,200mm) { $e_2$ };
\end{scope}
\end{tikzpicture} \; , \dots
\]
\noindent be the dual basis of $\XX(\lambda \otimes \nu,\mu)$. We call these diagrams {\bf trivalent vertices}. It is important to notice that trivalent vertices are not canonically defined.
\end{Definition}

\begin{Definition} \label{def:fusion_graph}
Pick a distinguished simple object $X \in \XX$. The {\bf fusion graph} of $X$ has vertices $\Lambda$ and the edges from $\lambda$ to $\mu$ are the distinguished basis vectors in $\XX(\mu,\lambda \otimes X)$.
\end{Definition}

\begin{Proposition} \label{prop:hom_space_tree_basis}
Fix $\lambda \in \Lambda$. Then $\XX(\lambda,X^{\otimes n})$ has dimension the number of paths from the tensor unit to $\lambda$ in the fusion graph for $X$ of length $n$. Moreover, an explicit basis is given by string diagrams of the form
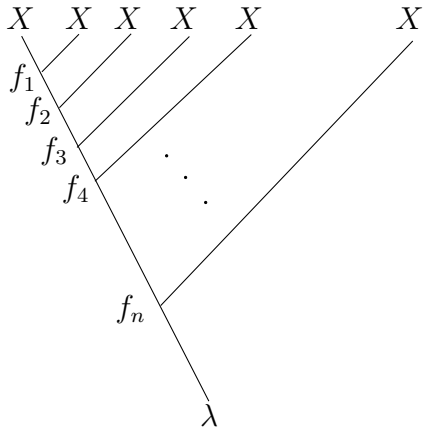
\[
\begin{tikzpicture}[yscale=-1,scale=0.03,baseline={([yshift=-.5ex]current bounding box.center)}]
\begin{scope}[shift={(0.00mm,719.29mm)}]
\draw [fill=black,draw=black] (659.15mm,-50.45mm) circle (4.80mm) ;
\draw [fill=black,draw=black] (743.15mm,45.55mm) circle (4.80mm) ;
\draw [fill=black,draw=black] (832.15mm,156.55mm) circle (4.80mm) ;
\draw [fill=none,draw=black] (270.72mm,-589.75mm)
-- ++(-164.05mm,166.88mm)
;
\draw [fill=none,draw=black] (502.00mm,-591.64mm)
-- ++(-318.00mm,328.00mm)
;
\draw [fill=none,draw=black] (16.66mm,-581.15mm)
-- (846.26mm,1037.66mm)
;
\draw [fill=none,draw=black] (262.24mm,-87.49mm)
-- ++(497.80mm,-492.15mm)
-- ++(0.00mm,0.00mm)
-- ++(0.00mm,0.00mm)
;
\draw [fill=none,draw=black] (341.44mm,62.41mm)
-- ++(692.96mm,-650.54mm)
-- ++(0.00mm,0.00mm)
;
\draw [fill=none,draw=black] (632.76mm,613.96mm)
-- (1749.99mm,-559.84mm)
;
\node [black] at (20.00mm,-395.64mm) { $f_1$ };
\node [black] at (90.00mm,-247.64mm) { $f_2$ };
\node [black] at (160.00mm,-71.64mm) { $f_3$ };
\node [black] at (256.00mm,96.36mm) { $f_4$ };
\node [black] at (500.00mm,628.36mm) { $f_n$ };
\node [black] at (12.00mm,-660mm) { $X$ };
\node [black] at (272.00mm,-660.64mm) { $X$ };
\node [black] at (488.00mm,-660.64mm) { $X$ };
\node [black] at (736.00mm,-660.64mm) { $X$ };
\node [black] at (1024.00mm,-660.64mm) { $X$ };
\node [black] at (1736.00mm,-660.64mm) { $X$ };
\node [black] at (848.00mm,1100mm) { $\lambda$ };
\end{scope}
\end{tikzpicture}
\]
\noindent In this diagram, each $f_i$ is a trivalent vertex of the form
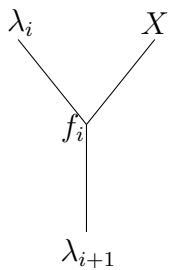
\[
\begin{tikzpicture}[yscale=-1,scale=0.015,baseline={([yshift=-.5ex]current bounding box.center)}]
\begin{scope}[shift={(0.00mm,719.29mm)}]
\draw [fill=none,draw=black] (887.14mm,53.08mm)
-- (281.43mm,-701.21mm)
;
\draw [fill=none,draw=black] (886.64mm,53.08mm)
-- (1492.36mm,-701.21mm)
;
\draw [fill=none,draw=black] (886.79mm,52.41mm)
-- ++(0.00mm,953.58mm)
;
\node [black] at (302.86mm,-850mm) { $\lambda_{i}$ };
\node [black] at (1494.29mm,-850mm) { $X$ };
\node [black] at (902.86mm,1200mm) { $\lambda_{i+1}$ };
\node [black] at (768.57mm,86.65mm) { $f_i$ };
\end{scope}
\end{tikzpicture}
\]
\noindent we call such string diagrams {\bf trivalent basis vectors}
\end{Proposition}
\begin{proof}
Decompose $X^{\otimes n}$ using the fusion graph for $X$.
\end{proof}
\begin{Definition} \label{def:matrix_units}
If $\XX$ is a semi-simple tensor category over $\mathbb{C}$ with finite dimensional morphism spaces, the Artin-Wedderburn Theorem implies that $\End(X^{\otimes n})$ is a product of matrix algebras. Proposition \ref{prop:hom_space_tree_basis} implies that in the trivalent basis, the matrix units in  $\End(X^{\otimes n})$ look like
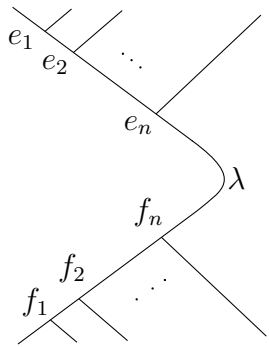
\[
\begin{tikzpicture}[yscale=-1,scale=0.03,baseline={([yshift=-.5ex]current bounding box.center)}]
\begin{scope}[shift={(0.00mm,719.29mm)}]
\draw [fill=none,draw=black] (250.00mm,-500.49mm)
-- ++(118.57mm,-100.00mm)
-- ++(0.00mm,0.00mm)
;
\draw [fill=none,draw=black] (377.14mm,-407.64mm)
-- (592.86mm,-593.35mm)
;
\draw [fill=none,draw=black] (743.12mm,-134.98mm)
-- ++(464.67mm,-438.41mm)
;
\draw [fill=none,draw=black] (107.08mm,-608.33mm)
-- ++(798.02mm,593.97mm)
;
\draw [fill=none,draw=black] (274.00mm,777.81mm)
-- ++(118.57mm,100.00mm)
-- ++(0.00mm,0.00mm)
;
\draw [fill=none,draw=black] (401.14mm,684.95mm)
-- ++(215.71mm,185.71mm)
;
\draw [fill=none,draw=black] (767.12mm,412.30mm)
-- ++(464.67mm,438.41mm)
;
\draw [fill=none,draw=black] (131.08mm,885.64mm)
-- (929.10mm,291.67mm)
;
\draw [fill=none,draw=black] (901.43mm,-17.05mm)
.. controls (1141.87mm,158.38mm) and (1031.26mm,204.81mm) .. (929.00mm,291.64mm)
;
\node [black] at (1100mm,150mm) { $\lambda$ };
\node [black] at (145.71mm,-464.78mm) { $e_1$ };
\node [black] at (297.14mm,-353.35mm) { $e_2$ };
\node [black] at (665.71mm,-84.78mm) { $e_n$ };
\node [black] at (208.34mm,700.86mm) { $f_1$ };
\node [black] at (370.78mm,550.34mm) { $f_2$ };
\node [black] at (709.90mm,300.67mm) { $f_n$ };
\draw [fill=black,draw=black] (596.18mm,-401.93mm) ellipse (2.65mm and 2.53mm) ;
\draw [fill=black,draw=black] (634.47mm,-373.36mm) ellipse (2.65mm and 2.53mm) ;
\draw [fill=black,draw=black] (673.76mm,-341.21mm) ellipse (2.65mm and 2.53mm) ;
\draw [fill=black,draw=black] (657.13mm,687.36mm) ellipse (2.65mm and 2.53mm) ;
\draw [fill=black,draw=black] (727.74mm,639.16mm) ellipse (2.65mm and 2.53mm) ;
\draw [fill=black,draw=black] (782.18mm,601.60mm) ellipse (2.65mm and 2.53mm) ;
\end{scope}
\end{tikzpicture}
\]
\noindent Equivalently, the irreducible representations of $\End(X^{\otimes n})$ are parameterized by the simple objects in $\XX$ which have a length $n$ path from the tensor unit in the fusion graph for $X$. The string diagrams defined in Proposition \ref{prop:hom_space_tree_basis} form a basis for the corresponding representation.
\end{Definition}

\begin{Definition} \label{ex:main_example}

The {\bf Iwahori-Hecke algebra}, denoted by $H_m$, is the algebra generated over $\mathbb{C}(a)$ by $1,g_1,\dots,g_{m-1}$ subject to the relations
\begin{align*}
g_i g_{i+1} g_i &= g_{i+1} g_i g_{i+1} \\
g_i g_j &= g_j g_i \qquad \text{if \quad  $\lvert i - j \lvert \geq 2$} \\
g_i^2 &= (a - a^{-1})g_i + 1.
\end{align*}
We define the category $H$ which has objects the natural numbers and morphisms
\[
H(m,n) =
\begin{cases}
H_m & m = n \\
0 & {\rm otherwise}.
\end{cases}
\]
The inclusion $H_m \otimes H_n \to H_{m+n}$ defined by $g_i \otimes g_j \mapsto g_i g_{m+j}$ equips $H$ with a tensor structure. We define $\mathcal{H} \subseteq [H^{\rm op},\Vect]$ to be the idempotent completion of $H$. The monoidal structure on $H$ extends to $\mathcal{H}$ via Day convolution. Morally, the category $\mathcal{H}$ can be described as finite dimensional type $1$ representations of the quantum group $U_a(\mathfrak{gl}_{\infty})$. The Grothendieck ring for $\mathcal{H}$ has basis given by partitions and multiplication given by the Littlewood-Richardson rule. A special case of the Littlewood-Richardson rule is the Pieri rule:
\[
\lambda \otimes \ydiagram{1} = \sum_{\lambda \subset \mu \vdash n+1} \mu
\]
This implies that the fusion graph for $\ydiagram{1}$ is Young's graph:
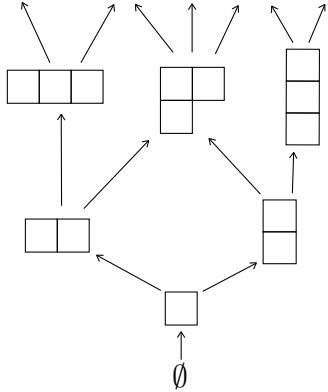
\[
\input{tikz_pictures/youngs_graph.tex}
\]
Paths in the Young graph are in bijection with standard partition fillings. It follows that the trivalent basis vectors
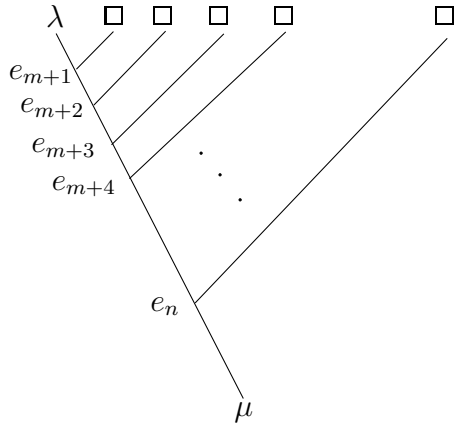
\[
\begin{tikzpicture}[yscale=-1,scale=0.03,baseline={([yshift=-.5ex]current bounding box.center)}]
\begin{scope}[shift={(0.00mm,719.29mm)}]
\draw [fill=black,draw=black] (659.15mm,-50.45mm) circle (4.80mm) ;
\draw [fill=black,draw=black] (743.15mm,45.55mm) circle (4.80mm) ;
\draw [fill=black,draw=black] (832.15mm,156.55mm) circle (4.80mm) ;
\draw [fill=none,draw=black] (270.72mm,-589.75mm)
-- ++(-164.05mm,166.88mm)
;
\draw [fill=none,draw=black] (502.00mm,-591.64mm)
-- ++(-318.00mm,328.00mm)
;
\draw [fill=none,draw=black] (16.66mm,-581.15mm)
-- (846.26mm,1037.66mm)
;
\draw [fill=none,draw=black] (262.24mm,-87.49mm)
-- ++(497.80mm,-492.15mm)
-- ++(0.00mm,0.00mm)
-- ++(0.00mm,0.00mm)
;
\draw [fill=none,draw=black] (341.44mm,62.41mm)
-- ++(692.96mm,-650.54mm)
-- ++(0.00mm,0.00mm)
;
\draw [fill=none,draw=black] (632.76mm,613.96mm)
-- (1749.99mm,-559.84mm)
;
\node [black] at (-50mm,-395.64mm) { $e_{m+1}$ };
\node [black] at (0.00mm,-247.64mm) { $e_{m+2}$ };
\node [black] at (50.00mm,-71.64mm) { $e_{m+3}$ };
\node [black] at (140.00mm,96.36mm) { $e_{m+4}$ };
\node [black] at (500.00mm,628.36mm) { $e_n$ };
\node [black] at (12.00mm,-660mm) { $\lambda$ };
\node [black] at (272.00mm,-660.64mm) { $\ydiagram{1}$ };
\node [black] at (488.00mm,-660.64mm) { $\ydiagram{1}$ };
\node [black] at (736.00mm,-660.64mm) { $\ydiagram{1}$ };
\node [black] at (1024.00mm,-660.64mm) { $\ydiagram{1}$ };
\node [black] at (1736.00mm,-660.64mm) { $\ydiagram{1}$ };
\node [black] at (848.00mm,1100mm) { $\mu$ };
\end{scope}
\end{tikzpicture}
\]
are in bijection (up to scaling) with standard skew tableaux of shape $\mu \backslash \lambda$. We abuse notation and identify these tree basis vectors with the corresponding standard skew tableaux. In \cite{MR1427801}, Ram and Leduc computed semi-normal forms for the Iwahori-Hecke algebras. More precisely, suppose that $\lambda \subseteq \mu \vdash n+2$ are partitions such that $\mu \backslash \lambda$ is not contained in a single row or column. Then there are exactly two partitions which satisfy $\lambda \subseteq \nu \subseteq \mu$. Call them $\nu$ and $\nu'$. The multiplicity space $\mathcal{H}(\mu,\lambda \otimes \ydiagram{1} \otimes \ydiagram{1})$ is 2-dimensional with basis
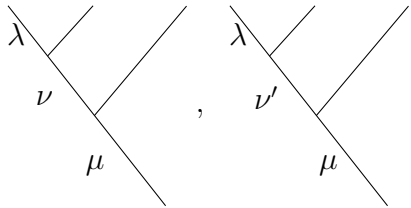
\[
\begin{tikzpicture}[yscale=-1,scale=0.015,baseline={([yshift=-.5ex]current bounding box.center)}]
\begin{scope}[shift={(0.00mm,719.29mm)}]
\draw [fill=none,draw=black] (125.71mm,-716.21mm)
-- (1522.86mm,1058.08mm)
;
\draw [fill=none,draw=black] (880.00mm,-716.21mm)
-- ++(-402.86mm,445.71mm)
-- ++(0.00mm,0.00mm)
;
\draw [fill=none,draw=black] (891.43mm,258.08mm)
-- (1708.57mm,-713.35mm)
;
\node [black] at (200mm,-461.92mm) { $\lambda$ };
\node [black] at (450mm,100mm) { $\nu$ };
\node [black] at (900mm,700mm) { $\mu$ };
\end{scope}
\end{tikzpicture} \; , \;
\begin{tikzpicture}[yscale=-1,scale=0.015,baseline={([yshift=-.5ex]current bounding box.center)}]
\begin{scope}[shift={(0.00mm,719.29mm)}]
\draw [fill=none,draw=black] (125.71mm,-716.21mm)
-- (1522.86mm,1058.08mm)
;
\draw [fill=none,draw=black] (880.00mm,-716.21mm)
-- ++(-402.86mm,445.71mm)
-- ++(0.00mm,0.00mm)
;
\draw [fill=none,draw=black] (891.43mm,258.08mm)
-- (1708.57mm,-713.35mm)
;
\node [black] at (200mm,-461.92mm) { $\lambda$ };
\node [black] at (450mm,100mm) { $\nu'$};
\node [black] at (1000mm,700mm) { $\mu$ };
\end{scope}
\end{tikzpicture}
\]
and $g_1$ acts via the matrix
\[
m(g_1) =
\begin{pmatrix}
a^{d}/[d] & [d-1][d+1]/[d]^2 \\
1 & a^{-d} / [-d]
\end{pmatrix}
\]
where
\[
[n] = \frac{a^n - a^{-n}}{a-a^{-1}}
\]
and $d = d_1 + d_2$ is the axial distance in $\mu \backslash \lambda$:
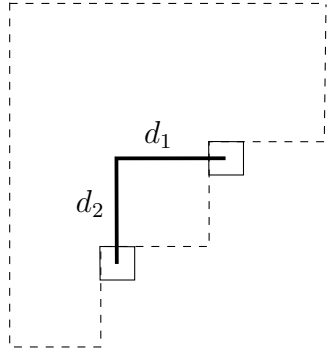
\[
\begin{tikzpicture}[yscale=-1,scale=0.04,baseline={([yshift=-.5ex]current bounding box.center)}]
\begin{scope}[shift={(0.00mm,719.29mm)}]
\draw [fill=none,draw=black,dashed] (162.86mm,-399.07mm)
-- ++(0.00mm,1145.71mm)
-- ++(302.86mm,0.00mm)
-- ++(0.00mm,-334.29mm)
-- ++(360.00mm,0.00mm)
-- ++(0.00mm,-348.57mm)
-- ++(382.86mm,0.00mm)
-- ++(5.71mm,-454.29mm)
;
\draw [fill=none,draw=black,dashed] (162.86mm,-396.21mm)
-- ++(1051.43mm,0.00mm)
-- ++(0.00mm,6.34mm)
;
\draw [fill=none,draw=black] (824.29mm,63.08mm)
-- (940.00mm,63.79mm)
-- ++(0.00mm,110.71mm)
-- ++(-115.00mm,0.00mm)
-- cycle
;
\draw [fill=none,draw=black] (462.65mm,412.45mm)
-- ++(115.71mm,0.71mm)
-- ++(0.00mm,110.71mm)
-- ++(-115.00mm,0.00mm)
-- cycle
;
\draw [fill=none,draw=black, line width = 0.5mm] (518.90mm,470.44mm)
-- (517.71mm,119.21mm)
-- (880.48mm,118.91mm)
;
\node [black] at (657.10mm,40.28mm) { $d_1$ };
\node [black] at (430.53mm,270.26mm) { $d_2$ };
\end{scope}
\end{tikzpicture}
\]
More formally, if $\mu \backslash \lambda$ contains the boxes $(a_1,b_1),(a_2,b_2)$, then the axial distance is defined by $d = \lvert a_1 - a_2 \lvert + \lvert b_1 - b_2 \lvert$. These formulas are quantum analogues of the well known Young semi-normal form for the representation theory of the symmetric group \cite{MR644144}. Indeed, when $a \to 1$, they recover they classical Young semi-normal formulas.

\end{Definition}

\section{Morita Theory} \label{sec:morita_theory}

In this section, we prove a very mild generalization of classical Morita theory. In classical Morita theory, we replace an object with its presentation with respect to a single projective. We are going to replace an object with its presentation with respect to a family of projectives. For the remainder of this section, $\XX$ is an abelian category enriched over $\Vect_k$, closed under colimits, $\DD$ is a category enriched over $\Vect_k$ and $D : \DD^{\rm op} \to \XX$ is a functor.

\begin{Theorem} \label{thm:morita_theory_main_thm}
If $\XX$ has enough projectives, then $\XX$ is equivalent to the category of representations of $\DD$ where $\DD^{\rm op}$ is a full subcategory of $\XX$ whose objects are compact, projective and generate $\XX$.
\end{Theorem}

We can prove this in a very clean way using coends. They can be motivated as follows: Suppose that $A$ is a $k$-algebra, $M$ is a left $A$-module and $N$ is a right $A$-module. Then we can form the tensor product $M \otimes_A N$ which is a vector space. It is built by taking the tensor product $M \otimes_k N$ and quotienting by the relations
$$ a m \otimes n = m \otimes n a. $$
We can generalize the second step in the following way. Suppose that $ F : \DD \otimes_k \DD^{\rm op} \to \Vect_k $ is a bifunctor. Then we can form the vector space
\[\int^{d \in \DD} F = \bigoplus_{d \in \DD} F(d,d) \; / \; fv = vf \qquad v \in F(d,d'), f : d' \to d.\] This vector space is called the {\bf coend} of $F$. We can use coends to generalize tensor products from modules to functors. Suppose that $F : \DD \to \Vect$ and $G : \DD^{\rm op} \to \Vect$ are functors. Then we define
\[
F \otimes_{\DD} G = \int^{d \in D} F(d) \otimes G(d).
\]
A clear exposition of the theory of coends can be found in \cite{MR3221774}. Let $\DD$ a category enriched over $\Vect$. Suppose that we have a functor $D : \DD^{\rm op} \to \XX$. Then we get a functor
\begin{align*}
&\XX \to [\DD,\Vect] \\
& X \mapsto \XX(D(-),X)
\end{align*}
This functor has a left adjoint given by
\begin{align*}
&[\DD,\Vect] \to \XX \\
&V \mapsto V \otimes_{\DD} D = \int^d V_d \otimes D^d
\end{align*}
The following computation demonstrates why these functors are adjoint:
\begin{align*}
\XX(V \otimes_{\DD} D,X) &= \XX \left( \int^d V_d \otimes D^d,X \right) \\
&= \int_d \XX(V_d \otimes D^d,X) \\
&= \int_d \hom(V_d,\XX(D^d,X)) \\
&= [\DD,\Vect](V,\XX(D(-),X))
\end{align*}
\begin{Definition}
  We call $X \in \XX$ a {\bf compact} object if $\XX(X,-)$ commutes with filtered colimits.
  \end{Definition}
\begin{Proposition} \label{prop:fully_faithful}
Assume that $D$ is fully faithful and each $D(d)$ is projective and compact. Then $[\DD,\Vect] \to \XX$ is fully faithful.
\end{Proposition}
\begin{proof}
We need to prove that the unit
$$V \to \XX(D(-),V \otimes_{\DD} D)$$
is an isomorphism. It suffices to prove this pointwise, so we need to prove that the linear map
$$ V(d) \to \XX(D(d),V \otimes_{\DD} D)$$
is an isomorphism. Since $D(d)$ is projective and compact, it follows that $\XX(D(d),-)$ commutes with all colimits. Therefore
\begin{align*}
\XX \left( D(d),\int^x V(x) \otimes D(x) \right) &= \int^x V(x) \otimes \XX(D(d),D(x)) \\
&= \int^x V(x) \otimes \DD(x,d) \\
&= V(d)
\end{align*}
The second equality is true because $\DD$ is fully faithful.
\end{proof}
\begin{Proposition} \label{prop:essentially_surjective}
In addition to the hypotheses of proposition \ref{prop:fully_faithful}, assume that every $X \in \XX$ admits an epimorphism $\bigoplus_i D(d_i) \to X$ for some family $\{ d_i\}$. Then $[\DD,\Vect] \to \XX$ is essentially surjective.
\end{Proposition}
\begin{proof}
By assumption, it follows that for every $X \in \XX$, the counit
$$ \XX(D(-),X) \otimes_{\DD} D \to X$$
is an epimorphism. Then we have an exact sequence
$$ 0 \to K \to \XX(D(-),X) \otimes_{\DD} D \to X \to 0$$
This gives us an exact sequence
$$ \XX(D(-),K) \otimes_{\DD} D \to \XX(D(-),X) \otimes_{\DD} D \to X \to 0$$
Since $-\otimes_{\DD} D$ is fully faithful, we can write the first map as $f \otimes_{\DD} D$ for some map $f : \XX(D(-),K) \to \XX(D(-),X)$. Since $- \otimes_{\DD} D$ is right exact, it follows that $X = \coker f \otimes_{\DD} D$. This proves essential surjectivity.
\end{proof}

\noindent {\em Proof of theorem \ref{thm:morita_theory_main_thm}.} Let $\DD^{\rm op}$ be a full subcategory of $\XX$ whose objects are compact, projective and generate $\XX$. Let $D : \DD^{\rm op} \to \XX$ be the embedding. By proposition, \ref{prop:fully_faithful}, the functor $ - \otimes_{\DD} D : [\DD,\Vect] \to \XX$ if fully faithful. By Proposition \ref{prop:essentially_surjective}, the functor is essentially surjective. \qed

\begin{Definition} \label{def:minimal_model}
  If $\XX$ is an abelian category with enough compact projectives, define $M(\XX)$ to be the opposite of the full subcategory with objects the indecomposable compact projectives. We call $M(\XX)$ the {\bf minimal model} for $\XX$. By theorem \ref{thm:morita_theory_main_thm}, the functor category $[M(\XX),\Vect]$ is equivalent to $\XX$.
\end{Definition}

\section{Modules over Tensor Algebras} \label{sec:modules_over_tensor_algebras}

In this section, we work inside a fixed semi-simple tensor category $\mathcal{C}$. We use Morita theory to study the category of modules over an algebra internal to $\mathcal{C}$. Choose a distinguished simple object $X \in \mathcal{C}$. Define
\[
T = \bigoplus_{n \geq 0} X^{\otimes n}
\]
This is the tensor algebra generated by $X$. Define $\Rep(T)$ to be the category of right modules over $T$ internal to $\mathcal{C}$. The forgetful functor $F : \Rep(T) \to \mathcal{C}$ has left adjoint $L : \mathcal{C} \to \Rep(T)$ defined by $V \mapsto V \otimes T$. Since the right adjoint $F$ is exact, it follows that $L$ preserves projectives. Define
\[
T^{+} = \bigoplus_{n \geq 1} X^{\otimes n}
\]

\begin{Lemma} \label{lem:indecomposable_projectives}
  If $V \in \mathcal{C}$ is irreducible, then $V \otimes T$ is an indecomposable projective in $\Rep(T)$.
\end{Lemma}
\begin{proof}
  Since $V \otimes T = L(V)$, the module is projective. Suppose that $V \otimes T = A \oplus B$ as $T$-modules. When we tensor with $T / T^+$, we get
  \[
  V = A / A T^+ \oplus B / B T^+
  \]
  in $\mathcal{C}$. Since $V$ is irreducible in $\mathcal{C}$, we can assume without loss of generality that $A / AT^{+} = 0$. Suppose that $A \not= 0$. Choose $0 \not= Y \subseteq A \subseteq V \otimes T$ irreducible in $\mathcal{C}$. This implies that
  \[
  Y \subseteq \bigoplus_{n=0}^N V \otimes X^{\otimes n}
  \]
  for some large $N$. Since $A = A \left( T^+ \right)^{N+1}$, it follows that
  \[ Y \subseteq A \left( T^+ \right)^{N+1} \subseteq \bigoplus_{n \geq N+1} V \otimes X^{\otimes n}. \]
  This implies that $Y = 0$, which is a contradiction. Therefore we must have $A = 0$.
\end{proof}
\begin{Proposition} Let $G$ be the fusion graph for $X$ considered as a category where the objects are vertices and the morphisms are paths. Then $\Rep(T)$ is Morita equivalent to $[G,\Vect]$.
\end{Proposition}

\begin{proof}
  The indecomposable compact projectives $\lambda \otimes T$, where $\lambda$ is an irreducible in $\mathcal{C}$, generate $\Rep(T)$. Using the adjunction $(L,F) : \Rep(T) \to \mathcal{C}$, we have
\[
\hom_T(\mu \otimes T,\lambda \otimes T) = \mathcal{C}(\mu,\lambda \otimes T).
\]
The right hand side has a basis consisting of vectors of the form
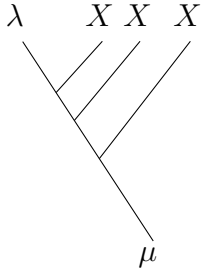
\[
\begin{tikzpicture}[yscale=-1,scale=0.015,baseline={([yshift=-.5ex]current bounding box.center)}]
\begin{scope}[shift={(0.00mm,719.29mm)}]
\draw [fill=none,draw=black] (148.57mm,-716.21mm)
-- (1305.71mm,1052.36mm)
;
\draw [fill=none,draw=black] (444.47mm,-263.87mm)
-- (854.65mm,-719.24mm)
;
\draw [fill=none,draw=black] (606.16mm,-16.48mm)
-- (1188.84mm,-719.24mm)
;
\draw [fill=none,draw=black] (827.72mm,321.87mm)
-- (1634.78mm,-719.11mm)
;
\node [black] at (80.84mm,-950mm) { $\lambda$ };
\node [black] at (829.21mm,-950mm) { $X$ };
\node [black] at (1168.47mm,-950mm) { $X$ };
\node [black] at (1610.14mm,-950mm) { $X$ };
\node [black] at (1258.65mm,1200mm) { $\mu$ };
\end{scope}
\end{tikzpicture}
\]
which is exactly a path in the fusion graph for $X$ from $\lambda$ to $\mu$. Post composing with the corresponding morphism in $\hom_T(\mu \otimes T,\lambda \otimes T)$ is the map
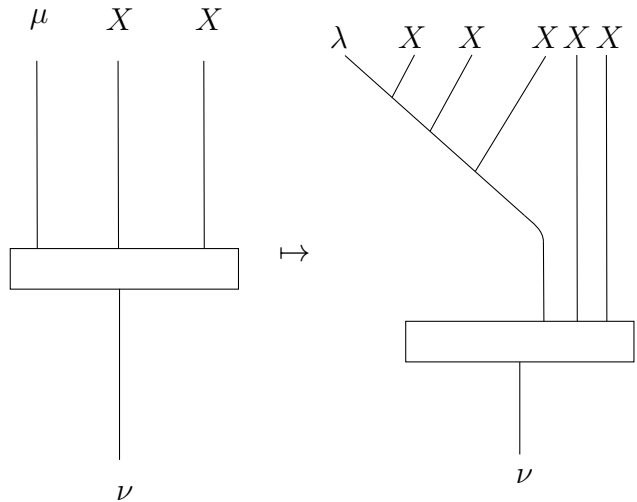
\[
\begin{tikzpicture}[yscale=-1,scale=0.03,baseline={([yshift=-.5ex]current bounding box.center)}]
\begin{scope}[shift={(0.00mm,719.29mm)}]
\draw [fill=none,draw=black] (342.86mm,115.22mm)
-- ++(0.00mm,180.00mm)
-- ++(1011.43mm,0.00mm)
-- ++(0.00mm,-180.00mm)
-- cycle
;
\draw [fill=none,draw=black] (827.14mm,293.79mm)
-- ++(0.00mm,758.57mm)
;
\draw [fill=none,draw=black] (462.71mm,115.64mm)
-- ++(-2.08mm,-833.06mm)
;
\draw [fill=none,draw=black] (822.71mm,115.64mm)
-- ++(-2.08mm,-833.06mm)
;
\draw [fill=none,draw=black] (1202.71mm,115.64mm)
-- ++(-2.08mm,-833.06mm)
;
\node [black] at (828.57mm,-900mm) { $X$ };
\node [black] at (1225.71mm,-900mm) { $X$ };
\node [black] at (471.43mm,-900mm) { $\mu$ };
\node [black] at (850mm,1200mm) { $\nu$ };
\end{scope}
\end{tikzpicture}
\quad \mapsto
\begin{tikzpicture}[yscale=-1,scale=0.03,baseline={([yshift=-.5ex]current bounding box.center)}]
\begin{scope}[shift={(0.00mm,719.29mm)}]
\draw [fill=none,draw=black] (321.64mm,461.81mm)
-- ++(0.00mm,180.00mm)
-- ++(1011.43mm,0.00mm)
-- ++(0.00mm,-180.00mm)
-- cycle
;
\draw [fill=none,draw=black] (827.14mm,641.79mm)
-- ++(0.00mm,410.57mm)
;
\draw [fill=none,draw=black] (933.94mm,462.28mm)
-- (931.86mm,109.23mm)
;
\draw [fill=none,draw=black] (1081.88mm,461.91mm)
-- ++(-2.08mm,-1181.11mm)
;
\draw [fill=none,draw=black] (1212.95mm,461.91mm)
-- ++(-2.08mm,-1181.11mm)
;
\draw [fill=none,draw=black] (931.97mm,109.26mm)
.. controls (931.28mm,56.48mm) and (884.03mm,25.30mm) .. (848.00mm,-7.41mm)
;
\draw [fill=none,draw=black] (848.25mm,-7.33mm)
-- (52.50mm,-716.95mm)
;
\draw [fill=none,draw=black] (360.62mm,-719.65mm)
-- ++(-99.57mm,188.49mm)
;
\draw [fill=none,draw=black] (615.18mm,-719.65mm)
-- (429.30mm,-380.90mm)
;
\draw [fill=none,draw=black] (941.87mm,-713.99mm)
-- ++(-312.43mm,511.79mm)
;
\node [black] at (848.00mm,1150mm) { $\nu$ };
\node [black] at (25.00mm,-800mm) { $\lambda$ };
\node [black] at (350.10mm,-800mm) { $X$ };
\node [black] at (616.43mm,-800mm) { $X$ };
\node [black] at (941.25mm,-800mm) { $X$ };
\node [black] at (1080.63mm,-800mm) { $X$ };
\node [black] at (1226.97mm,-800mm) { $X$ };
\end{scope}
\end{tikzpicture}
\]
This implies that composition of basis vectors is exactly concatenation of paths in the fusion graph for $X$. This completes the proof.
\end{proof}

\begin{Example} \label{ex:minimal_model_tensor_algebra_first_fundamental}
Let $\mathcal{C} = \mathcal{H}$, which was defined in Example \ref{ex:main_example}, and let $X = \ydiagram{1}$. The fusion graph for $X$ has objects partitions and the edges $G(\lambda,\mu)$ are the standard skew tableaux of shape $\mu \backslash \lambda$.
\end{Example}

\section{Modules over the quantum symmetric algebra} \label{sec:modules_over_the_quantum_symmetric_algebra}

In this section, we work inside the category $\mathcal{H}$ defined in Example \ref{ex:main_example}. Define $T = \bigoplus_{n \geq 0} \ydiagram{1}^{\otimes n}$. Consider the submodule $I$ of $T$ spanned by all maps
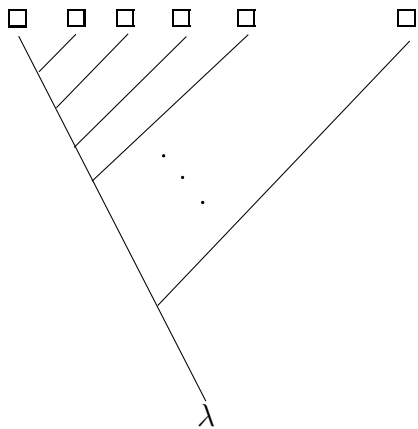
\[
\begin{tikzpicture}[yscale=-1,scale=0.03,baseline={([yshift=-.5ex]current bounding box.center)}]
\begin{scope}[shift={(0.00mm,719.29mm)}]
\draw [fill=black,draw=black] (659.15mm,-50.45mm) circle (4.80mm) ;
\draw [fill=black,draw=black] (743.15mm,45.55mm) circle (4.80mm) ;
\draw [fill=black,draw=black] (832.15mm,156.55mm) circle (4.80mm) ;
\draw [fill=none,draw=black] (270.72mm,-589.75mm)
-- ++(-164.05mm,166.88mm)
;
\draw [fill=none,draw=black] (502.00mm,-591.64mm)
-- ++(-318.00mm,328.00mm)
;
\draw [fill=none,draw=black] (16.66mm,-581.15mm)
-- (846.26mm,1037.66mm)
;
\draw [fill=none,draw=black] (262.24mm,-87.49mm)
-- ++(497.80mm,-492.15mm)
-- ++(0.00mm,0.00mm)
-- ++(0.00mm,0.00mm)
;
\draw [fill=none,draw=black] (341.44mm,62.41mm)
-- ++(692.96mm,-650.54mm)
-- ++(0.00mm,0.00mm)
;
\draw [fill=none,draw=black] (632.76mm,613.96mm)
-- (1749.99mm,-559.84mm)
;
\node [black] at (12.00mm,-660mm) { $\ydiagram{1}$ };
\node [black] at (272.00mm,-660.64mm) { $\ydiagram{1}$ };
\node [black] at (488.00mm,-660.64mm) { $\ydiagram{1}$ };
\node [black] at (736.00mm,-660.64mm) { $\ydiagram{1}$ };
\node [black] at (1024.00mm,-660.64mm) { $\ydiagram{1}$ };
\node [black] at (1736.00mm,-660.64mm) { $\ydiagram{1}$ };
\node [black] at (848.00mm,1100mm) { $\lambda$ };
\end{scope}
\end{tikzpicture}
\]
where $\lambda$ is a partition with two or more rows. The grading on the Grothendieck ring implies that $I$ is a 2-sided ideal in $T$, so we can form the quotient algebra $S = T/I$. We have
\[
S = \emptyset \oplus \ydiagram{1} \oplus \ydiagram{2} \oplus \ydiagram{3} \oplus \cdots
\]
Define $\Rep(S)$ to be category of right modules over $S$ internal to the category $\mathcal{H}$. Just like the tensor algebra, every projective $S$-module is free and the indecomposable projectives are of the form $\lambda \otimes S$ where $\lambda$ is a partition. Define $F$ to be the fusion graph for $\ydiagram{1}$ inside $\mathcal{H}$ interpreted as a category. Define $M$ to be the category whose objects are partitions and whose morphisms are defined by
\[
M(\lambda,\mu) = \hom_S(\mu \otimes S, \lambda \otimes S).
\]
Then we have the functor $Q = - \otimes_T S : F \to M$. By definition, this functor is the identity on objects. Since all the projectives involved are free, it follows that $Q$ is full. We can describe $Q$ more concretely as follows. Each hom space in $F$ is a skew representation of some Iwahori-Hecke algebra. We have:
\begin{Lemma} \label{lem:Q_is_projection_onto_invariants}
On morphisms, $Q$ projects onto the Hecke algebra invariants.
\end{Lemma}
 \begin{proof}Recall that given a vector $f \in F(\lambda,\mu)$, post composition by the induced map $\hom_T(\mu \otimes T, \lambda \otimes T)$ is given by
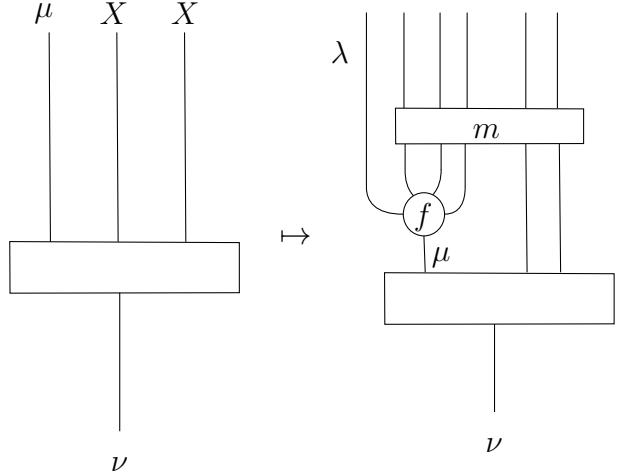
\[
\begin{tikzpicture}[yscale=-1,scale=0.03,baseline={([yshift=-.5ex]current bounding box.center)}]
\begin{scope}[shift={(0.00mm,719.29mm)}]
\draw [fill=none,draw=black] (348.32mm,211.45mm)
-- ++(1.19mm,228.76mm)
-- ++(1016.21mm,-2.02mm)
-- ++(0.00mm,-222.23mm)
-- cycle
;
\draw [fill=none,draw=black] (822.86mm,-716.21mm)
-- ++(5.71mm,928.57mm)
;
\draw [fill=none,draw=black] (1122.86mm,-716.21mm)
-- ++(5.71mm,928.57mm)
;
\draw [fill=none,draw=black] (522.86mm,-716.21mm)
-- ++(5.71mm,928.57mm)
;
\draw [fill=none,draw=black] (835.71mm,438.08mm)
-- ++(0.00mm,614.29mm)
;
\node [black] at (500mm,-800mm) {$\mu$ };
\node [black] at (810.14mm,-800mm) { $X$ };
\node [black] at (1125.56mm,-800mm) { $X$ };
\node [black] at (834.39mm,1200mm) { $\nu$ };
\end{scope}
\end{tikzpicture} \quad \mapsto
\begin{tikzpicture}[yscale=-1,scale=0.03,baseline={([yshift=-.5ex]current bounding box.center)}]
\begin{scope}[shift={(0.00mm,719.29mm)}]
\draw [fill=none,draw=black] (348.32mm,434.22mm)
-- ++(1.19mm,228.76mm)
-- ++(1016.21mm,-2.02mm)
-- ++(0.00mm,-222.23mm)
-- cycle
;
\draw [fill=none,draw=black] (974.28mm,-137.79mm)
-- ++(5.71mm,570.15mm)
;
\draw [fill=none,draw=black] (1122.87mm,-137.92mm)
-- ++(5.71mm,570.27mm)
;
\draw [fill=none,draw=black] (522.83mm,270.95mm)
-- ++(5.76mm,161.41mm)
;
\draw [fill=none,draw=black] (835.71mm,662.36mm)
-- ++(0.00mm,390.00mm)
;
\node [black] at (600mm,367.02mm) { $\mu$ };
\node [black] at (834.39mm,1200mm) { $\nu$ };
\draw [fill=none,draw=black] (522.86mm,175.22mm) ellipse (91.43mm and 95.71mm) ;
\draw [fill=none,draw=black] (476.35mm,92.62mm)
.. controls (448.68mm,61.90mm) and (439.00mm,21.77mm) .. (439.48mm,-20.52mm)
;
\draw [fill=none,draw=black] (562.16mm,88.45mm)
.. controls (589.83mm,57.74mm) and (599.51mm,17.61mm) .. (599.03mm,-24.68mm)
;
\draw [fill=none,draw=black] (614.93mm,174.07mm)
.. controls ++(69.06mm,0.86mm) and ++(0.53mm,63.21mm) .. ++(90.66mm,-120.46mm)
;
\draw [fill=none,draw=black] (431.38mm,174.52mm)
.. controls (306.97mm,175.39mm) and (267.09mm,117.36mm) .. (268.05mm,54.24mm)
;
\draw [fill=none,draw=black] (267.95mm,54.77mm)
-- ++(0.09mm,-774.79mm)
;
\draw [fill=none,draw=black] (705.59mm,54.21mm)
-- ++(0.09mm,-191.78mm)
;
\draw [fill=none,draw=black] (439.41mm,-19.97mm)
-- ++(0.09mm,-117.22mm)
;
\draw [fill=none,draw=black] (598.98mm,-22.46mm)
-- ++(0.09mm,-114.88mm)
;
\draw [fill=none,draw=black] (396.68mm,-295.55mm)
-- ++(0.98mm,158.30mm)
-- ++(832.83mm,-1.40mm)
-- ++(0.00mm,-153.78mm)
-- cycle
;
\node [black] at (520mm,185.44mm) { $f$ };
\node [black] at (799.84mm,-195.05mm) { $m$ };
\draw [fill=none,draw=black] (434.37mm,-295.18mm)
-- ++(0.00mm,-424.26mm)
;
\draw [fill=none,draw=black] (594.37mm,-295.18mm)
-- ++(0.00mm,-424.26mm)
;
\draw [fill=none,draw=black] (714.37mm,-295.18mm)
-- ++(0.00mm,-424.26mm)
;
\draw [fill=none,draw=black] (974.37mm,-295.18mm)
-- ++(0.00mm,-424.26mm)
;
\draw [fill=none,draw=black] (1114.37mm,-295.18mm)
-- ++(0.00mm,-424.26mm)
;
\node [black] at (150mm,-544.78mm) { $\lambda$ };
\end{scope}
\end{tikzpicture}
\]
More precisely, the map $f : \mu \to \lambda \otimes X^{\otimes n}$ induces a map $\mu \otimes T \to \lambda \otimes T$ defined by
\[
g : \mu \otimes T \xrightarrow{f \otimes 1} \lambda \otimes X^{\otimes n} \otimes T \xrightarrow{1 \otimes m} \lambda \otimes T
\]
where $m$ is the multiplication map. The diagram depicts post composing a map $\nu \to \mu \otimes T$ with $g$. By Yoneda's lemma, this determines $g$.  If we tensor along the projection $p : T \to S$ we have
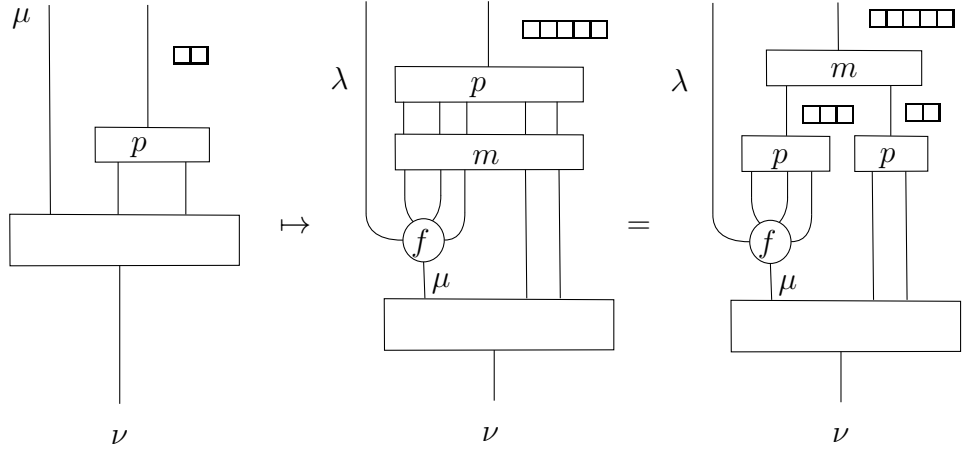
\[
\begin{tikzpicture}[yscale=-1,scale=0.03,baseline={([yshift=-.5ex]current bounding box.center)}]
\begin{scope}[shift={(0.00mm,719.29mm)}]
\draw [fill=none,draw=black] (348.32mm,211.45mm)
-- ++(1.19mm,228.76mm)
-- ++(1016.21mm,-2.02mm)
-- ++(0.00mm,-222.23mm)
-- cycle
;
\draw [fill=none,draw=black] (828.55mm,-20.13mm)
-- ++(0.00mm,232.49mm)
;
\draw [fill=none,draw=black] (1128.59mm,-20.49mm)
-- ++(0.00mm,232.84mm)
;
\draw [fill=none,draw=black] (522.86mm,-716.21mm)
-- ++(5.71mm,928.57mm)
;
\draw [fill=none,draw=black] (835.71mm,438.08mm)
-- ++(0.00mm,614.29mm)
;
\node [black] at (400mm,-658.38mm) { $\mu$ };
\node [black] at (834.39mm,1200mm) { $\nu$ };
\draw [fill=none,draw=black] (726.68mm,-174.48mm)
-- ++(0.59mm,154.90mm)
-- ++(505.81mm,-1.37mm)
-- ++(0.00mm,-150.48mm)
-- cycle
;
\node [black] at (921.43mm,-87.64mm) { $p$ };
\draw [fill=none,draw=black] (960.48mm,-172.13mm)
-- ++(-0.67mm,-544.65mm)
-- ++(0.00mm,0.00mm)
;
\node [black] at (1150mm,-493.35mm) { $\ydiagram{2}$ };
\end{scope}
\end{tikzpicture} \quad \mapsto
\begin{tikzpicture}[yscale=-1,scale=0.03,baseline={([yshift=-.5ex]current bounding box.center)}]
\begin{scope}[shift={(0.00mm,719.29mm)}]
\draw [fill=none,draw=black] (348.32mm,600.22mm)
-- ++(1.19mm,228.76mm)
-- ++(1016.21mm,-2.02mm)
-- ++(0.00mm,-222.23mm)
-- cycle
;
\draw [fill=none,draw=black] (974.28mm,28.21mm)
-- ++(5.71mm,570.15mm)
;
\draw [fill=none,draw=black] (1122.87mm,28.08mm)
-- ++(5.71mm,570.27mm)
;
\draw [fill=none,draw=black] (522.83mm,436.95mm)
-- ++(5.76mm,161.41mm)
;
\draw [fill=none,draw=black] (835.71mm,828.08mm)
-- ++(0.00mm,224.29mm)
;
\node [black] at (600mm,533.02mm) { $\mu$ };
\node [black] at (820mm,1200mm) { $\nu$ };
\draw [fill=none,draw=black] (522.86mm,341.22mm) ellipse (91.43mm and 95.71mm) ;
\draw [fill=none,draw=black] (476.35mm,258.62mm)
.. controls (448.68mm,227.90mm) and (439.00mm,187.77mm) .. (439.48mm,145.48mm)
;
\draw [fill=none,draw=black] (562.16mm,254.45mm)
.. controls ++(27.67mm,-30.72mm) and ++(0.48mm,42.29mm) .. ++(36.87mm,-113.14mm)
;
\draw [fill=none,draw=black] (614.93mm,340.07mm)
.. controls ++(69.06mm,0.86mm) and ++(0.53mm,63.21mm) .. ++(90.66mm,-120.46mm)
;
\draw [fill=none,draw=black] (431.38mm,340.52mm)
.. controls (306.97mm,341.39mm) and (267.09mm,283.36mm) .. (268.05mm,220.24mm)
;
\draw [fill=none,draw=black] (267.95mm,220.77mm)
-- ++(0.09mm,-939.45mm)
;
\draw [fill=none,draw=black] (705.59mm,220.21mm)
-- ++(0.09mm,-191.78mm)
;
\draw [fill=none,draw=black] (439.41mm,146.03mm)
-- ++(0.09mm,-117.22mm)
;
\draw [fill=none,draw=black] (598.98mm,143.54mm)
-- ++(0.09mm,-114.88mm)
;
\draw [fill=none,draw=black] (396.68mm,-129.55mm)
-- ++(0.98mm,158.30mm)
-- ++(832.83mm,-1.40mm)
-- ++(0.00mm,-153.78mm)
-- cycle
;
\node [black] at (509.12mm,351.44mm) { $f$ };
\node [black] at (799.84mm,-29.05mm) { $m$ };
\draw [fill=none,draw=black] (434.37mm,-129.18mm)
-- ++(0.00mm,-144.26mm)
;
\draw [fill=none,draw=black] (594.37mm,-129.18mm)
-- ++(0.00mm,-144.26mm)
;
\draw [fill=none,draw=black] (714.37mm,-129.18mm)
-- ++(0.00mm,-144.26mm)
;
\draw [fill=none,draw=black] (974.37mm,-129.18mm)
-- ++(0.00mm,-144.26mm)
;
\draw [fill=none,draw=black] (1114.37mm,-129.18mm)
-- ++(0.00mm,-144.26mm)
;
\node [black] at (150mm,-378.78mm) { $\lambda$ };
\draw [fill=none,draw=black] (396.68mm,-429.55mm)
-- ++(0.98mm,158.30mm)
-- ++(832.83mm,-1.40mm)
-- ++(0.00mm,-153.78mm)
-- cycle
;
\node [black] at (765.71mm,-339.07mm) { $p$ };
\draw [fill=none,draw=black] (810.00mm,-428.35mm)
-- ++(0.00mm,-290.71mm)
;
\node [black] at (1150mm,-594.19mm) { $\ydiagram{5}$ };
\end{scope}
\end{tikzpicture} =
\begin{tikzpicture}[yscale=-1,scale=0.03,baseline={([yshift=-.5ex]current bounding box.center)}]
\begin{scope}[shift={(0.00mm,719.29mm)}]
\draw [fill=none,draw=black] (348.32mm,600.22mm)
-- ++(1.19mm,228.76mm)
-- ++(1016.21mm,-2.02mm)
-- ++(0.00mm,-222.23mm)
-- cycle
;
\draw [fill=none,draw=black] (974.28mm,28.21mm)
-- ++(5.71mm,570.15mm)
;
\draw [fill=none,draw=black] (1122.87mm,28.08mm)
-- ++(5.71mm,570.27mm)
;
\draw [fill=none,draw=black] (522.83mm,436.95mm)
-- ++(5.76mm,161.41mm)
;
\draw [fill=none,draw=black] (835.71mm,828.08mm)
-- ++(0.00mm,224.29mm)
;
\node [black] at (600mm,533.02mm) { $\mu$ };
\node [black] at (834.39mm,1200mm) { $\nu$ };
\draw [fill=none,draw=black] (522.86mm,341.22mm) ellipse (91.43mm and 95.71mm) ;
\draw [fill=none,draw=black] (476.35mm,258.62mm)
.. controls (448.68mm,227.90mm) and (439.00mm,187.77mm) .. (439.48mm,145.48mm)
;
\draw [fill=none,draw=black] (562.16mm,254.45mm)
.. controls ++(27.67mm,-30.72mm) and ++(0.48mm,42.29mm) .. ++(36.87mm,-113.14mm)
;
\draw [fill=none,draw=black] (614.93mm,340.07mm)
.. controls ++(69.06mm,0.86mm) and ++(0.53mm,63.21mm) .. ++(90.66mm,-120.46mm)
;
\draw [fill=none,draw=black] (431.38mm,340.52mm)
.. controls (306.97mm,341.39mm) and (267.09mm,283.36mm) .. (268.05mm,220.24mm)
;
\draw [fill=none,draw=black] (267.95mm,220.77mm)
-- ++(0.09mm,-939.45mm)
;
\draw [fill=none,draw=black] (705.59mm,220.21mm)
-- ++(0.09mm,-191.78mm)
;
\draw [fill=none,draw=black] (439.41mm,146.03mm)
-- ++(0.09mm,-117.22mm)
;
\draw [fill=none,draw=black] (598.98mm,143.54mm)
-- ++(0.09mm,-114.88mm)
;
\node [black] at (509.12mm,351.44mm) { $f$ };
\draw [fill=none,draw=black] (594.37mm,-129.18mm)
-- ++(0.00mm,-221.41mm)
;
\draw [fill=none,draw=black] (1056.37mm,-129.18mm)
-- ++(0.00mm,-222.84mm)
;
\node [black] at (120mm,-378.78mm) { $\lambda$ };
\draw [fill=none,draw=black] (505.22mm,-509.59mm)
-- ++(0.81mm,158.37mm)
-- ++(687.35mm,-1.40mm)
-- ++(0.00mm,-153.85mm)
-- cycle
;
\node [black] at (1150mm,-650mm) { $\ydiagram{5}$ };
\draw [fill=none,draw=black] (396.56mm,-129.67mm)
-- ++(0.46mm,158.54mm)
-- ++(389.12mm,-1.40mm)
-- ++(0.00mm,-154.01mm)
-- cycle
;
\draw [fill=none,draw=black] (896.54mm,-129.69mm)
-- ++(0.38mm,158.59mm)
-- ++(323.53mm,-1.40mm)
-- ++(0.00mm,-154.06mm)
-- cycle
;
\draw [fill=none,draw=black] (822.26mm,-720.46mm)
-- ++(2.02mm,213.14mm)
-- ++(0.00mm,0.00mm)
;
\node [black] at (565.69mm,-26.48mm) { $p$ };
\node [black] at (1042.48mm,-30.52mm) { $p$ };
\node [black] at (850mm,-420mm) { $m$ };
\node [black] at (780mm,-224.47mm) { $\ydiagram{3}$ };
\node [black] at (1200mm,-232.55mm) {$\ydiagram{2}$ };
\end{scope}
\end{tikzpicture}
\]
The second equality is true because $p$ is an algebra homomorphism, so it commutes with multiplication.
\end{proof}

\begin{Proposition} \label{prop:skew_rep_invariants}
Suppose that $\lambda \subseteq \mu$ are partitions. Then  $\mathcal{H}(\mu,\lambda \otimes \ydiagram{1}^{\otimes n})$ has Hecke algebra invariants if and only if $\mu \backslash \lambda$ is a horizontal strip. In this case, the invariants are 1-dimensional and any skew tableaux projects onto a nonzero invariant.
\end{Proposition}
\begin{proof}
The invariants in $\mathcal{H}(\mu,\lambda \otimes \ydiagram{1}^{\otimes n})$ are the same as maps
\[
\mu \to \lambda \otimes \underbrace{\ydiagram{5}}_{n}
\]
By Pieri's rule, $\mathcal{H}(\mu,\lambda \otimes \ydiagram{1}^{\otimes n})$ has invariants if and only if $\mu \backslash \lambda$ is a horizontal strip. Suppose that $\mu \backslash \lambda$ is a horizontal strip and $P$ is a skew tableaux of shape $\mu \backslash \lambda$. Then from the semi-normal form, we know that $P$ generates $\mathcal{H}(\mu,\lambda \otimes \ydiagram{1}^{\otimes n})$. This implies that $\mathcal{H}(\mu,\lambda \otimes \ydiagram{1}^{\otimes n})$ has a $1$-dimensional space of invariants and $P$ projects onto a nonzero invariant.
\end{proof}

\noindent {\em Proof of Theorem \ref{thm:minimal_model}.} The minimal model for $S$ is $M$. From Lemma \ref{lem:Q_is_projection_onto_invariants}, $Q : F \to M$ is a full functor  which is projection onto the Hecke algebra invariants. From Proposition \ref{prop:skew_rep_invariants}, we have
$$M(\mu,\lambda) = F(\mu,\lambda)^{H_n} = \mathcal{H}(\mu,\lambda \otimes \ydiagram{1})^{H_n} = {\bf HS}(\mu,\lambda).$$
\qed

\bibliographystyle{alpha}
\bibliography{bibliography}

\end{document}